\documentclass{amsart}
\usepackage{amsmath,amssymb,bm,amscd}
\usepackage{color}
\usepackage[dvipdfmx]{graphicx}
\usepackage{here}
\numberwithin{equation}{section}
\theoremstyle{definition}
\newtheorem{defn}{Definition}[section]

\newtheorem{rem}[defn]{Remark}
\newtheorem{question}[defn]{Question}
\theoremstyle{plain}
\newtheorem{thm}[defn]{Theorem}

\newtheorem{lem}[defn]{Lemma}
\newtheorem{cor}[defn]{Corollary}

\newtheorem{conj}[defn]{Conjecture}

\title{On the Lagrangian fillability of almost positive links}
\author{Keiji Tagami}
\subjclass[2010]{57M25, 57R17}
\keywords{knot; almost positive knot; Legendrian knot; Lagrangian filling; contact structure}
\date{\today}
\address{
Department of Mathematics, Faculty of Science and Technology, Tokyo University
of Science, 2641 Yamazaki, Noda, Chiba, 278-8510, Japan
}
\email{tagami\_keiji@ma.noda.tus.ac.jp}
\begin{document}
\maketitle
\begin{abstract}
In this paper, we prove that a link which has an almost positive diagram with a certain condition is Lagrangian fillable. 
%
\end{abstract}
\section{Introduction}
A \textit{knot} is a smooth embedding of a circle into $\mathbf{R}^3$ and a \textit{link} is a smooth embedding of disjoint circles into $\mathbf{R}^3$. 
An oriented link is {\it positive} if it has a link diagram whose crossings are all positive. 
An oriented link is {\it almost positive} if it is not positive and has a link diagram with exactly one negative crossing. 
Such a diagram is called an {\it almost positive diagram}. 
It is known that almost positive links have many properties similar to those of positive links (for example, see \cite{homogeneous}, \cite{negative_signature}, \cite{almost_negative_signatire}, \cite{stoimenow1} and \cite{tagami5}). 
For this reason, in general, it is hard to distinguish positive links from almost positive links. 
\par
In \cite{Hayden-Sabloff}, Hayden and Sabloff studied positive knots in the light of contact and symplectic topology. 
In particular, they considered Lagrangian fillings of links in the symplectisation of the standard contact $3$-manifold $(\mathbf{R}^{3}, \xi _{std})$ and showed the following. 
For the definition of Lagrangian fillings, see Section~\ref{sec:concordance}. 
\begin{thm}[{\cite[Theorem 1.1]{Hayden-Sabloff}}]\label{thm:Hayden-Sabloff}
All positive links are exact Lagrangian fillable. 
\end{thm}
Naturally, we can consider the following question.  
\begin{question}
Is any almost positive link exact Lagrangian fillable? 
\end{question}
Here, we recall Hayden and Sabloff's observation on Lagrangian fillability \cite{Hayden-Sabloff}. 
By the results of Eliashberg \cite{Eliashberg1}, a Lagrangian fillable knot is isotopic to a transverse knot with a symplectic filling. 
By the work of Boileau and Orevkov \cite{Boileau-Orevkov}, we see that such a knot is quasipositive. 
Moreover, an exact Lagrangian filling of a Legendrian knot induces a $2$-graded normal ruling of the knot. 
(In fact, by Ekholm \cite{Ekholm1}, and Ekholm, Honda and K{\'a}lm{\'a}n \cite{Ekholm-Honda-Kalman}, an exact Lagrangian filling induces an (ungraded) augmentation. 
Because our Lagrangian fillings are oriented, such augmentations are $2$-graded (see \cite[Remark~2.3]{Cronwell-Ng-Sivek}). 
By Fuchs and Ishkhanov \cite{Fuchs-Ishkhanov} and Sabloff \cite{Sabloff1}, it is shown that the existence of a $2$-graded augmentation is equivalent to that of a $2$-graded normal ruling.)
By Rutherford's work \cite{Rutherford1}, for such a Legendrian knot, the HOMFLYPT bound on the maximal Thurston-Bennequin number is sharp, that is, $\operatorname{TB}(K)=-\max\operatorname{deg}_{v}P_{K}(v, z)-1$, where $\operatorname{TB}(K)$ is the maximal Thurston-Bennequin number of $K$ and $P_{K}(v, z)$ is the HOMFLYPT polynomial of $K$. 
In \cite{Hayden-Sabloff}, Hayden and Sabloff conjectured the following.  
\begin{conj}[{\cite[Conjecture 1.3]{Hayden-Sabloff}}]\label{conj:Hayden-Sabloff}
A knot is exact Lagrangian fillable if and only if it is quasipositive and the HOMFLYPT bound on the maximal Thurston-Bennequin number of $K$ is sharp. 
\end{conj}
On the other hand, the following are known: 
\begin{itemize}
\item if $K$ is Lagrangian fillable, then $\operatorname{TB}(K)=2g_4(K)-1$, where $g_{4}(K)$ is the $4$-ball genus of $K$. Moreover $g_4(K)$ is equal to the genus of its Lagrangian filling \cite{Chantraine}, 
%
\item $tb(\Lambda)+|r(\Lambda)|\leq 2g_{4}(K)-1\leq 2g_3(K)-1$, where $\Lambda$ is a Legendrian representative of $K$, $g_3(K)$ is the genus of $K$, $tb(\Lambda )$ is the Thurston-Bennequin number and $r(\Lambda)$ is the rotation number of $\Lambda$ \cite{Bennequin, Rudolph},  
\item $tb(\Lambda)+|r(\Lambda)|\leq 2\tau(K)-1$, where $\tau(K)$ is the Ozsv{\'a}th-Szab{\'o} invariant of $K$ \cite{Plamenevskaya}, 
\item $tb(\Lambda)+|r(\Lambda)|\leq s(K)-1 $, where $s(K)$ is the Rasmussen invariant of $K$ \cite{Plamenevskaya2, Shumakovitch}, 
\item $tb(\Lambda)+|r(\Lambda)|\leq -\max \operatorname{deg}_{v} P_{K}(v,z)-1$ \cite{Frank-Williams, Morton1} (see also \cite{Fuchs-Tabachnikov}). 
\end{itemize}
It is well known that if $K$ is quasipositive, we see that $s(K)$ and $2\tau(K)$ are equal to $2g_4(K)$ (\cite{Shumakovitch} for $s$ and \cite{Plamenevskaya} for $\tau$). 
Hence, we obtain the following. 
\begin{cor}\label{cor:fillable-equality}
If a knot $K$ is exact Lagrangian fillable, then $K$ is quasipositive and satisfies 
\[
\operatorname{TB}(K)+1=2\tau(K)=s(K)=2g_4(K)=-\max \operatorname{deg}_{v} P_{K}(v,z). 
\]
\end{cor}
\begin{rem}
The Lagrangian fillability implies $r(\Lambda )=0$. 
When $r(\Lambda )=0$, it is known that the sharpness of the HOMFLYPT bound induces the sharpness of the Kauffman bound on $tb(\Lambda )$ \cite{Rutherford1}, and the sharpness of the Rasmussen bound induces the sharpness of the Khovanov bound on $tb(\Lambda )$ \cite{Ng1}. 
\end{rem}
In this paper, we prove Theorem~\ref{thm:fillable1} below. 
\begin{thm}\label{thm:fillable1}
Let $D$ be an almost positive link diagram of a link $L$. 
Suppose that there is a positive crossing connecting the two Seifert circles which are connected by the negative crossing. 
Then $L$ is exact Lagrangian fillable. 
\end{thm}
In this paper, the condition supposed in Theorem~\ref{thm:fillable1} is called $($P$2)$. 
%
%
%


%
%
\par
Hayden-Sabloff \cite{Hayden-Sabloff} have proved that Lagrangian fillability and strongly quasipositivity are independent conditions. 
In particular, they gave a Lagrangian fillable and non-strongly quasipositive knot. 
In Section~\ref{sec:infinite}, we give infinitely many almost positive (in particular, non-positive), Lagrangian fillable and strongly quasipositive knots (Theorem~\ref{thm:almost-lagrangian}). 
%
\par
%
%
%
%
This paper is organized as follows: 
In Section~\ref{sec:concordance}, we recall the definition of Lagrangian fillings. 
In Section~\ref{sec:bunching}, we recall the bunching deformation, which is a key tool to prove the main result. 
In Section~\ref{sec:main}, we prove Theorem~\ref{thm:fillable1} (Theorem~\ref{lem:fillable1}). 
In Section~\ref{sec:infinite}, we give infinitely many almost positive, Lagrangian fillable and strongly quasipositive knots. 
In Section~\ref{sec:discussion}, we compare the Lagrangian fillability and the positivity of links. 
\par
Throughout this paper, we suppose that links and Legendrian links are oriented. 
In our pictures, the $y$-coordinate is the horizontal coordinate and the $z$-coordinate is the vertical coordinate. 
\section{Lagrangian fillings}\label{sec:concordance}
In this section, we recall the definition of Lagrangian fillings and describe a tool which allows us to construct Lagrangian fillings. 
\par
The \textit{standard contact structure} $\xi _{std}$ on $\mathbf{R}^{3}$ is $\operatorname{Ker} \alpha$, where $\alpha=dz+xdy$. 
A \textit{Legendrian link} in $(\mathbf{R}^3, \xi _{std})$ is a smooth embedding of disjoint circles which are tangent to $\xi _{std}$. 
A \textit{front projection} of a Legendrian link is the image of the link under the $(y, z)$-projection. 
A Legendrian link $\Lambda $ is a \textit{Legendrian representative} of a link $L$ if $\Lambda $ is isotopic to $L$ in smooth category. 
The \textit{symplectisation} of $(\mathbf{R}^3, \xi _{std})$ is the symplectic $4$-manifold $(\mathbf{R}\times \mathbf{R}^3, d(e^t\alpha))$, where $t$ is the first coordinate. 
Let $\Lambda _{0}$ and $\Lambda _{1}$ be oriented Legendrian links in $(\mathbf{R}^3, \xi _{std})$. 
Let $\Sigma$ be an oriented Lagrangian submanifold in the symplectisation, that is, an oriented $2$-submanifold with $d(e^t\alpha)|_{\Sigma}=0$. 
Then, $\Sigma$ is a \textit{Lagrangian cobordism} from $\Lambda _{0}$ to $\Lambda _{1}$ with cylindrical Legendrian ends $\mathcal{E}_{\pm }$ if there exists a pair of real numbers $T_{-}<T_{+}$ such that 
\begin{itemize}
\item $\mathcal{E}_{+}:=\Sigma \cap (T_{+}, \infty)\times \mathbf{R}^{3}=(T_{+}, \infty)\times \Lambda _{1}, $
\item $\mathcal{E}_{-}:= \Sigma \cap (-\infty, T_{-})\times \mathbf{R}^{3}=(-\infty, T_{-})\times \Lambda _{0}, $ and
\item $\Sigma\setminus (\mathcal{E}_{+}\cup \mathcal{E}_{-})$ is a compact oriented surface with boundary $\Lambda _{1}\cup (-\Lambda _{0})$. 
\end{itemize}
Moreover, if $e^{t}\alpha |_{\Sigma}$ is exact and $f$ is constant on each of $\mathcal{E}_{\pm }$ whenever $df=e^{t}\alpha|_{\Sigma}$, we call $\Sigma$ an {\it exact} Lagrangian cobordism. 
If there exists a Lagrangian cobordism $\Sigma$ from $\Lambda _{0}$ to $\Lambda _{1}$, we say $\Lambda _{0}$ is \textit{Lagrangian cobordant} to $\Lambda _{1}$ (denoted by $\Lambda _0 \prec_{\Sigma} \Lambda _1$). 
An oriented Legendrian link $\Lambda $ is \textit{Lagrangian fillable} if $\emptyset \prec _{\Sigma} \Lambda $. 
Then $\Sigma$ is called a \textit{Lagrangian filling} of $\Lambda$. 
A smooth oriented link is \textit{Lagrangian fillable} if it  has a Legendrian representative with a Lagrangian filling (see \cite{Chantraine}). 
Similarly, exact Lagrangian cobordisms, exact Lagrangian fillablility and exact Lagrangian fillings are defined. 
\par
Here, we introduce tools to construct (exact) Lagrangian cobordisms. 
\begin{thm}[{\cite[Theorem 2.2]{Hayden-Sabloff}, \cite{BFS, Chantraine, Ekholm-Honda-Kalman, Rizell}}]\label{thm:tool}
Let $\Lambda_0$ and $\Lambda_1$ be Legendrian links in $(\mathbf{R}^3, \xi _{std})$. 
If one of the following holds, we obtain $\Lambda_0\prec_{\Sigma} \Lambda_1$ with an exact Lagrangian cobordism $\Sigma$. 
\begin{itemize}
\item[Isotopy:] $\Lambda_0$ and $\Lambda_1$ are Legendrian isotopic. 
\item[$0$-handle:] a front projection of $\Lambda_1$ is a disjoint union of a front projection of $\Lambda_0$ and a front projection of the Legendrian unknot with $tb=-1$ and $rot=0$ (see the left picture in Figure~\ref{fig:handle}). 
\item[$1$-handle:] a front projection of $\Lambda_1$ and a front projection of $\Lambda_0$ are related as the right picture in Figure~\ref{fig:handle}. 
\end{itemize}
\end{thm}
\begin{figure}[htbp]
\begin{center}
\includegraphics[scale=0.4]{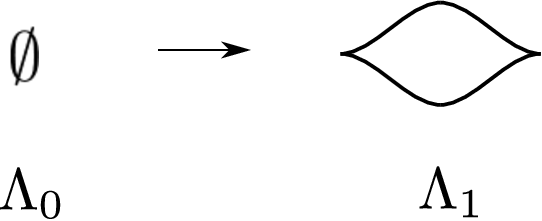}
\hspace{40pt}
\includegraphics[scale=0.4]{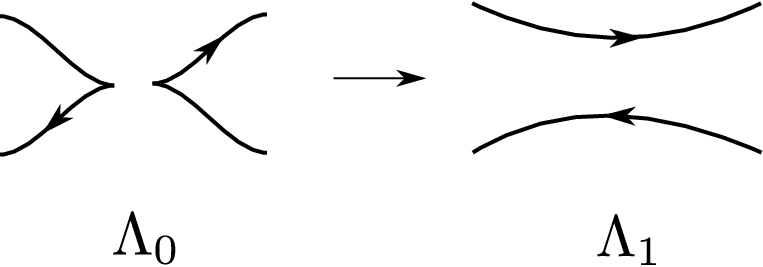}
\end{center}
\caption{A $0$-handle attaching (left). A $1$-handle attaching (right). }
\label{fig:handle}
\end{figure}
\begin{lem}\label{lem:seifert}
Let $\Delta $ be a front projection of a Legendrian link $\Lambda $. 
Let $\Gamma $ be a Seifert circle of $\Delta $. 
Suppose that $\Gamma $ satisfies the following: 
\begin{itemize}
\item $\Gamma $ is an innermost Seifert circle of $\Delta $, 
\item every crossing adjacent to $\Gamma $ has both strands oriented downward or upward with respect to the $y$-coordinate as the top picture in Figure~\ref{fig:removing} (in particular, it is positive crossing), and 
\item $\Gamma $ has exactly one left cusp and one right cusp (in particular, they are the local minimum and local maximum of $\Gamma$ with respect to the $y$-coordinate). 
\end{itemize}
Let $\Lambda' $ be the Legendrian link which has the front projection obtained from $\Delta $ by removing $\Gamma $ and its adjacent crossings. 
Then, $\Lambda \succ_{\Sigma} \Lambda '$ with an exact Lagrangian cobordism $\Sigma$. 
\end{lem}
\begin{proof}
This proof is essentially due to Hayden and Sabloff \cite{Hayden-Sabloff}. 
Let $c$ be the number of the crossings adjacent to $\Gamma $. 
We prove by induction on $c$. 
If $c=0$, by Theorem~\ref{thm:tool} ($0$-handle attaching), we obtain $\Lambda \succ_{\Sigma} \Lambda '$. 
Suppose that Lemma~\ref{lem:seifert} is true if $c<k$. 
Let $c=k$. 
Let $\Delta ''$ be the front projection  of a Legendrian link $\Lambda ''$ obtained by removing the lowest (positive) crossing adjacent to $\Gamma $ with respect to the $y$-coordinate. 
Then, by Figure~\ref{fig:removing}, we see that $\Lambda \succ_{\Sigma'} \Lambda ''$. 
By the induction hypothesis, $\Lambda'' \succ_{\Sigma''} \Lambda '$. 
Hence, $\Lambda \succ_{\Sigma} \Lambda '$. 
In this proof, we only use Theorem~\ref{thm:tool}. 
Hence, Lagrangian cobordisms are all exact. 
\begin{figure}[h]
\begin{center}
\includegraphics[scale=0.4]{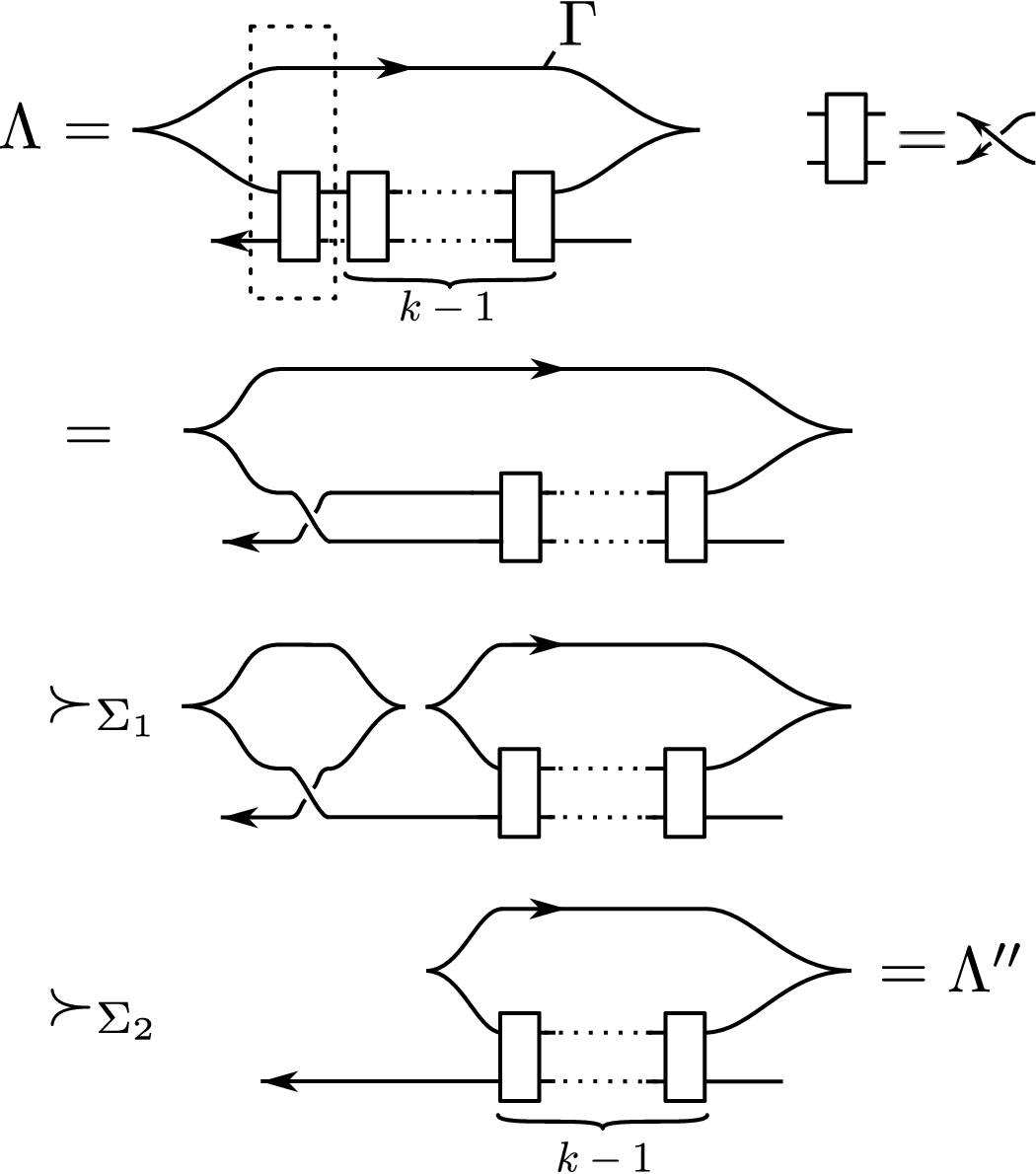}
\end{center}
\caption{An exact Legendrian cobordism from $\Lambda ''$ to $\Lambda $. In the third line, we use a $1$-handle attaching. In the fourth line, we use a Legendrian isotopy. In this picture, for simplicity, we suppose that there is no crossing oriented upward. In the case where there are some crossings oriented upward, we can construct a Legendrian cobordism similarly. }
\label{fig:removing}
\end{figure}
\end{proof}
\section{Bunching deformation}\label{sec:bunching}
In this section, we recall an operation called \textit{bunching deformation} \cite{S-Yamada}. 
\par
Two disjoint oriented circles on $\mathbf{S}^2=\mathbf{R}^2\cup \{\infty\}$ are {\it coherent} if they are homologous on $A$, where $A$ is the annulus bounded by the circles on $\mathbf{S}^2$.  
Let $D$ be a link diagram, and $C_1$ and $C_{2}$ be distinct Seifert circles of $D$. 
Suppose that $C_1$ and $C_{2}$ are not coherent and there is a band $b$ on $\mathbf{S}^2$ such that 
$b\cap D=\partial b\cap (C_1\cup C'_2)=d_1\cup d_2$, where $C'_2$ is a slight large copy of $C_{2}$, $d_1$ is a subarc of $C_1$ and $d_2$ is a subarc of $C'_2$. 
Put $C'_1=C_{1}\cup C'_2\cup \partial b\setminus (d_1\cup d_2)$.  
Then, we call the operation replacing $C_1$ with $C'_1$ by a {\it bunching deformation along $b$} (see Figure~\ref{fig:bunching}). 
This deformation corresponds to the ``bunching operation of type II" \cite{S-Yamada}. 
It is well known that by using the bunching deformation, Yamada \cite{S-Yamada} proved that the minimal number of Sifert circles of a link equals the minimal braid index of the link.  
By utilizing this deformation, Tanaka \cite{T-Tanaka} found a Legendrian representative of a positive link which attains the maximal Thurston-Bennequin number.  
In order to prove our results, we apply Tanaka's technique to almost positive diagrams. 
\begin{figure}[h]
\begin{center}
\includegraphics[scale=0.3]{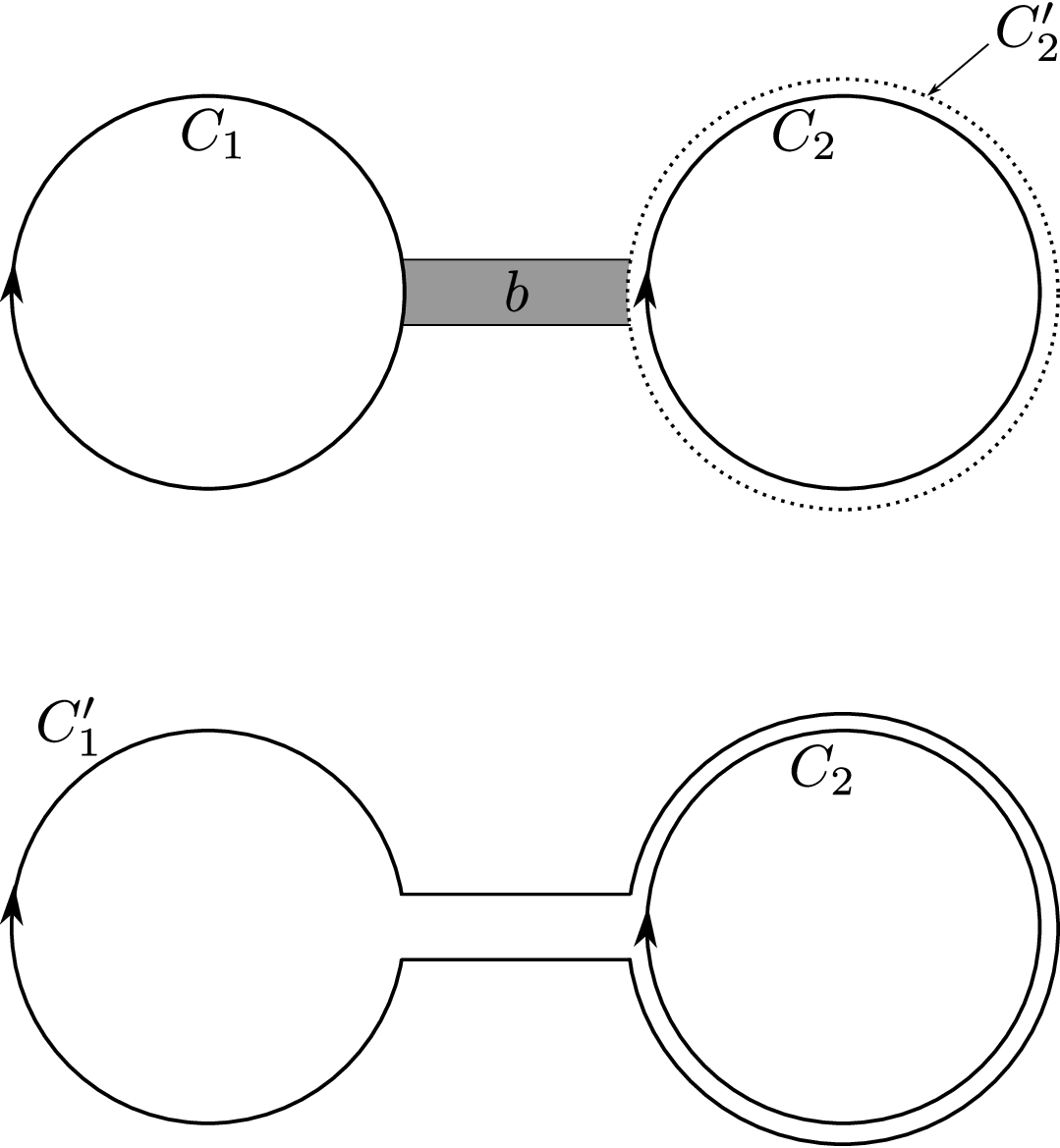}
\end{center}
\caption{Bunching deformation}
\label{fig:bunching}
\end{figure}
\begin{lem}\label{lem:deformation}
Let $D$ be an almost positive diagram with the negative crossing $p$. 
Then we can deform $D$ into an almost positive diagram $D'$ on the $(y,z)$-plane satisfying the following: 
\begin{itemize}
\item[(1)] each crossing is oriented downward with respect to the $y$-coordinate, 
\item[(2)] each Seifert circle has exactly one local maximum and one local minimum with respect to the $y$-coordinate, 
\item[(3)] the negative crossing $p$ is the highest crossing with respect to the $y$-coordinate, 
\item[(4)] the two Seifert circles connected by the negative crossing $p$ are not nested. 
\end{itemize}
\end{lem}
\begin{proof}
Let $D$ be an almost positive diagram with the negative crossing $p$. 
Put $D$ on the $(y,z)$-plane so that the two Seifert circles connected by the negative crossing $p$ are not nested and so that they are outermost Seifert circles  (see (iii) in Figure~\ref{fig:bunching3}). 
Then, connect $p$ and the point at infinity by a path $l$ on $\mathbf{S}^2=\mathbf{R}^2\cup \{\infty\}$ (see (iv) in Figure~\ref{fig:bunching3}). 
\par 
For the diagram, apply bunching deformations until one can, where 
\begin{itemize}
\item the bands used in the bunching deformations are on $\mathbf{S}^2\setminus l$ and  
\item 
if we need to apply a bunching deformation appearing one of the two Seifert circles connected by $p$, denote the Seifert circle by $S_{p}$, we apply the bunching deformation so that $S_p$ is outermost. In other words, $S_p$ plays a role of $C_1$ in the definition of the bunching deformation. 
\end{itemize}
Note that in the resulting diagram, the two Seifert circles connected by $p$ are the only outermost Seifert circles 
(see (v) in Figure~\ref{fig:bunching3}, which is obtained from (iv) by applying bunching deformations along the red dotted arcs). 
Here, we draw the subrarcs $d_1$ which are used in the bunching deformations as blue dotted arcs (in (v)-(viii) in Figure~\ref{fig:bunching3}, we draw the subarcs by the blue dotted arcs). 
\par 
Then, by isotopy on $\mathbf{S}^2\setminus l$, we can deform the diagram so that it is presented by the closure of a braid (see (vi) in Figure~\ref{fig:bunching3}). 
Notice that the closure is taken in $\mathbf{S}^2\setminus l$. 
After the isotopy, the negative crossing $p$ may not be the highest crossing with respect to the $y$-coordinate (the horizontal coordinate). 
In that case, by taking an appropriate conjugate for the braid, deform the diagram so that $p$ is the highest crossing with respect to the $y$-coordinate (see (vii) in Figure~\ref{fig:bunching3}). 
\par
Now this diagram satisfies $(1)$--$(4)$ but it is not almost positive. 
If we can apply the inverse of the bunching deformations which preserves $(1)$--$(4)$, we finish the proof. 
However, in general, the blue dotted arcs $d_1$ have some local maxima and minima with respect to the $y$-coordinate, and the inverse of the bunching deformations along the blue dotted arcs $d_1$ does not preserve $(2)$. 
So, we delete these maxima and minima as follows: 
for each blue dotted arc $d_1$, deform $d_1$ so that the two endpoints lie on the braid, and consider the disk bounded by the union of $d_1$ and the line segment connecting the two endpoints of $d_1$ (see the gray area in (vii) in Figure~\ref{fig:bunching3}). 
Firstly, we take an outermost one of such disks and shrink the disk by an isotopy which fixes the line segment until the corresponding blue dotted arc $d_1$ has no local maximum and local minimum with respect to the $y$-coordinate (see (viii) in Figure~\ref{fig:bunching3}). 
Here, ``outermost" means the disk is not contained in other disks. 
Note that such an isotopy does not deform the diagram at the outside of the disk. 
Secondly, we take an outermost one of the remaining disks and deform similarly.  
By repeating this deformation inductively, we delete all local maxima and local minima of all blue dotted arcs. 
\par 
Finally, by the inverse of the bunching deformations along the blue dotted arcs, we obtain the desired diagram (see (ix) in Figure~\ref{fig:bunching3}). 
\if{
Put $D$ on the $(y,z)$-plane so that 
the two Seifert circles connected by the negative crossing $p$ are not nested 
and $p$ is the highest crossing with respect to the $y$-coordinate. 
Then, we can connect $p$ and the point at infinity by a path $l$ on $\mathbf{S}^2=\mathbf{R}^2\cup \{\infty\}$. 
By applying bunching deformations and the inverses of the bunching deformations as in \cite{T-Tanaka} on $\mathbf{R}^2\setminus l$, we obtain the desired diagram $D'$. 
For example, see Figure~\ref{fig:bunching2}. 
In Figure~\ref{fig:bunching2}, we draw a crossing by a rectangle. 
In the third picture, we deform the diagram so that the two Seifert circles connected by the negative crossing $p$ are not nested. 
In the fourth picture, we rotate and deform the diagram so that the negative crossing is the highest crossing with respect to the $y$-coordinate (the horizontal coordinate). 
By applying bunching deformations along the red dotted lines in the fourth picture, we obtain the fifth picture. 
The blue dotted lines in the fifth and sixth pictures are marks for the inverse of the bunching deformations. 
By an isotopy, we obtain a braid presentation for the link (the sixth picture). 
Here, we note that all crossings are oriented downward with respect to the $y$-coordinate. 
Finally, by the inverse of the bunching deformations, we obtain the desired diagram. 
}\fi
\end{proof}
\begin{figure}[hp]
\begin{center}
\includegraphics[scale=0.673]{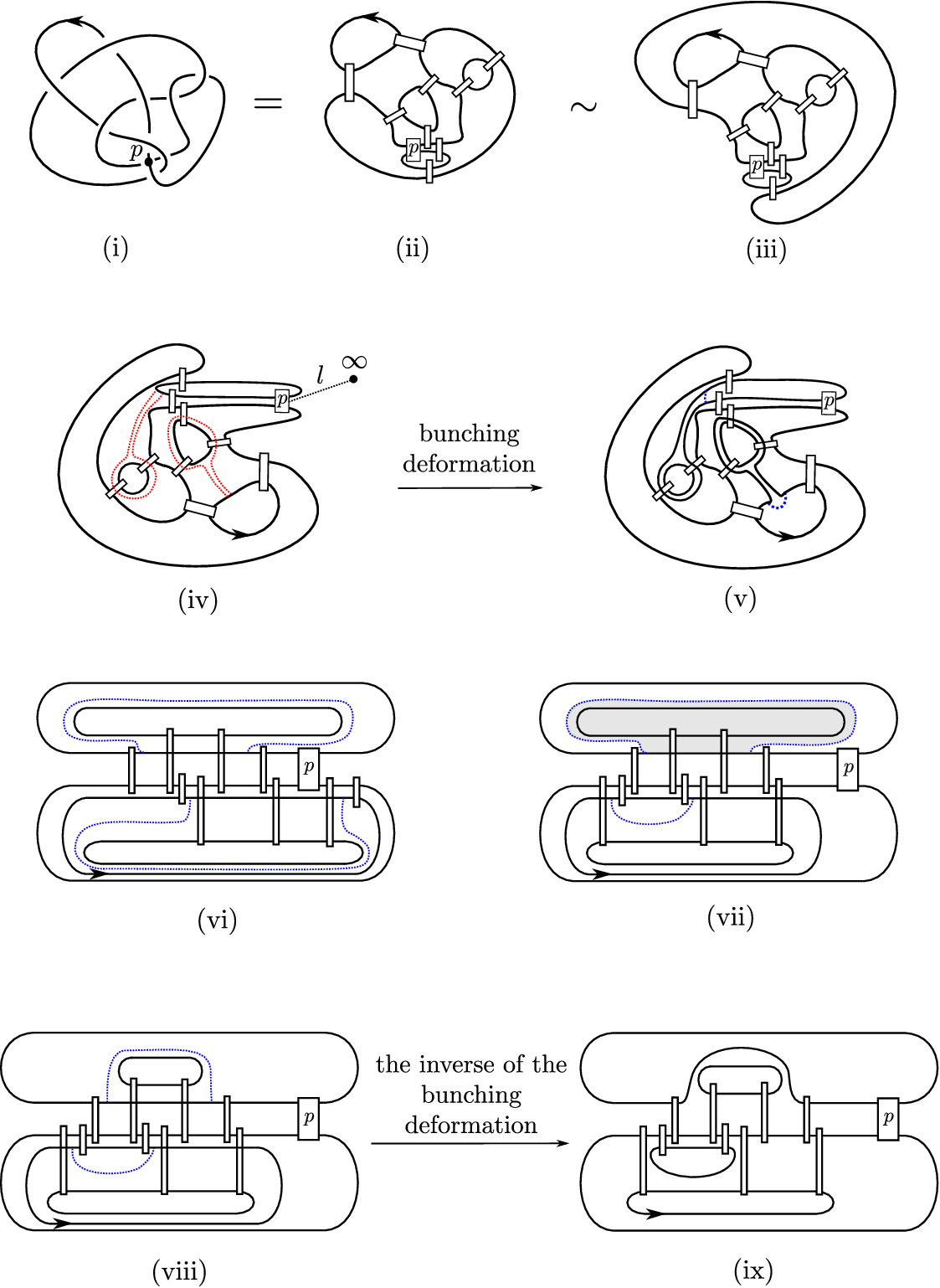}
\end{center}
\caption{(color online) Sketch for the proof of Lemma~\ref{lem:deformation}. The picture (i) is $D$. In pictures (ii)-(ix), we draw a crossing derived from a crossing of $D$ by a rectangle. }
\label{fig:bunching3}
\end{figure}
%
\section{The Lagrangian fillability of almost positive links}\label{sec:main}
In this section, we consider the Lagrangian fillability of almost positive links. 
\par
Let $D$ be an almost positive link diagram of a link $L$ with negative crossing $p$. 
Then, $D$ satisfies one of the following properties: 
\begin{itemize}
\item[(P$1$):] there is no positive crossing joining the two Seifert circles which are connected by $p$ (see the left of Figure~\ref{fig:picture1}), 
\item[(P$2$):] there is a positive crossing joining the two Seifert circles which are connected by $p$ (see the right of Figure~\ref{fig:picture1}). 
\end{itemize}
In \cite{stoimenow1}, Stoimenow considered the two properties and computed the genera of almost positive knots. 
In \cite{tagami5}, the author also considered these properties and computed the Rasmussen invariants and $4$-ball genera of almost positive knots (see also \cite{abe-tagami1}). 
By the following result, we see that if $D$ satisfies (P2), then $L$ is Lagrangian fillable. 
\begin{figure}[h]
\begin{center}
\includegraphics[scale=0.7]{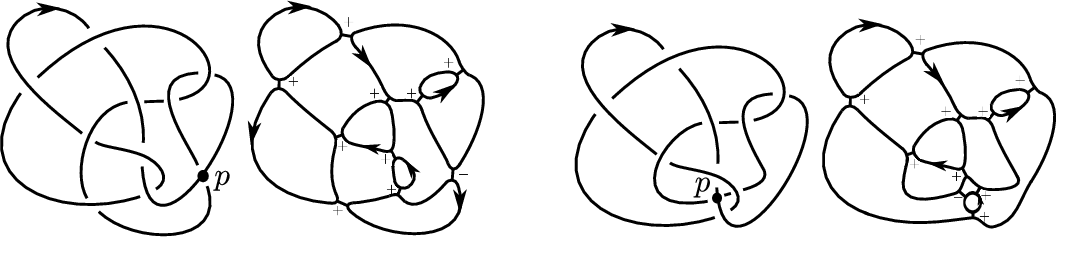}
\end{center}
\caption{Two almost positive diagrams of $10_{145}$, which is an almost positive knot. 
The left diagram satisfies (P$1$). 
The right diagram satisfies (P$2$). }
\label{fig:picture1}
\end{figure}
\begin{thm}[Theorem~\ref{thm:fillable1}]\label{lem:fillable1}
Let $D$ be an almost positive link diagram of a link $L$ with negative crossing $p$ with the property $(P2)$. 
Then $L$ is exact Lagrangian fillable. 
\end{thm}
\begin{proof}
By Lemma~\ref{lem:deformation}, the diagram $D$ can be transformed into an almost positive diagram $D'$ on the $(y,z)$-plane satisfying the following: 
\begin{itemize}
\item[(0)] the property (P$2$), 
\item[(1)] each crossing is oriented downward with respect to the $y$-coordinate, 
\item[(2)] each Seifert circle has exactly one local maximum and one local minimum with respect to the $y$-coordinate, 
\item[(3)] the negative crossing $p$ is the highest crossing with respect to the $y$-coordinate. 
\item[(4)] the two Seifert circles connected by the negative crossing $p$ are not nested. 
\end{itemize}
We remark that $D'$ satisfies $(0)$ because $D$ and $D'$ are isotopic on $\mathbf{S}^2$ (see the proof of Lemma~\ref{lem:deformation}). 
In order to construct a Legendrian representative of $L$ with an exact Lagrangian filling, firstly, we deform $D'$ near the negative crossing as in Figure~\ref{fig:negative-crossing}. 
After this deformation, one of the Seifert circles connected by the negative crossing does not satisfy the condition $(2)$ anymore.  
Next, we deform $D'$ near each local maximum or local minimum as in Figure~\ref{fig:cusp}. 
Then, we obtain a front projection $\Delta $ of a Legendrian representative $\Lambda $ of $L$. 
\par
Let $s$ be the number of the  Seifert circles of $\Delta $. 
We prove that $\Lambda $ is exact Lagrangian fillable by the induction on $s$. 
\par
If $s=2$, the front projection $\Delta $ is as the left in Figure~\ref{fig:s=2}. 
By the Legendrian version of Reidemeister move II, we obtain another front projection $\Delta '$ of $\Lambda $ from $\Delta$ as the right in Figure~\ref{fig:s=2}. 
By Lemma~\ref{lem:seifert} (or \cite[Proof of Theorem 1.1 or Remark~3.2]{Hayden-Sabloff}), we see that the Legendrian link with the front projection $\Delta '$, that is $\Lambda $, is  exact Lagrangian fillable. 
\par
Suppose $s\geq 3$. 
Then, we can suppose that there exists a Seifert circle $\Gamma $ of $\Delta $ such that it is an innermost circle and it is not adjacent to the negative crossing. 
In fact, if there is no such circle, all Seifert circles except the two Seifert circles connected by the negative crossing contain the two circles. 
In that case, we can remove the negative crossing by the Legendrian version of Reidemeister move II and prove that $\Lambda $ is exact Lagrangian fillable by the same discussion as the case $s=2$. 
Otherwise, such a Seifert circle $\Gamma $ satisfies the assumption of Lemma~\ref{lem:seifert}. 
Hence, $\Lambda \succ_{\Sigma}\Lambda ' $, where $\Lambda '$ is a Legendrian link which has the front projection obtained from $\Delta $ by removing $\Gamma $ and its adjacent crossings. 
Note that this Lagrangian cobordism is exact. 
By the induction hypothesis, $\Lambda'$ is exact Lagrangian fillable, and so is $\Lambda$. 
\end{proof}
\begin{figure}[h]
\begin{center}
\includegraphics[scale=0.4]{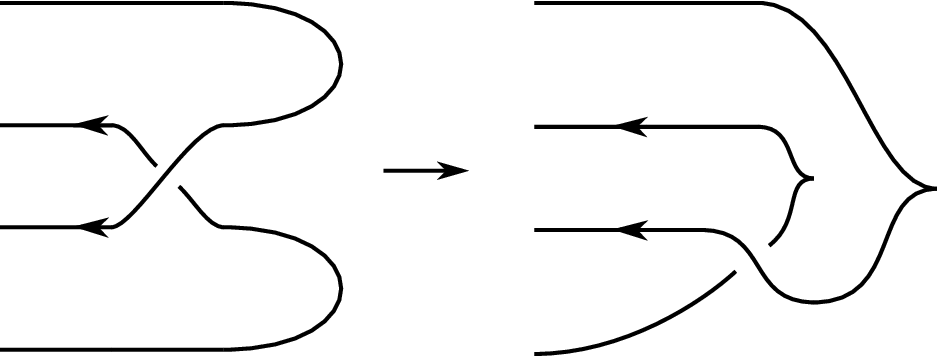}
\end{center}
\caption{Deformation near the negative crossing}
\label{fig:negative-crossing}
\end{figure}
\begin{figure}[h]
\begin{center}
\includegraphics[scale=0.5]{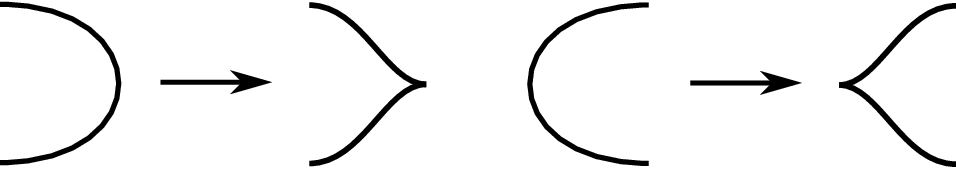}
\end{center}
\caption{}
\label{fig:cusp}
\end{figure}
\begin{figure}[h]
\begin{center}
\includegraphics[scale=0.65]{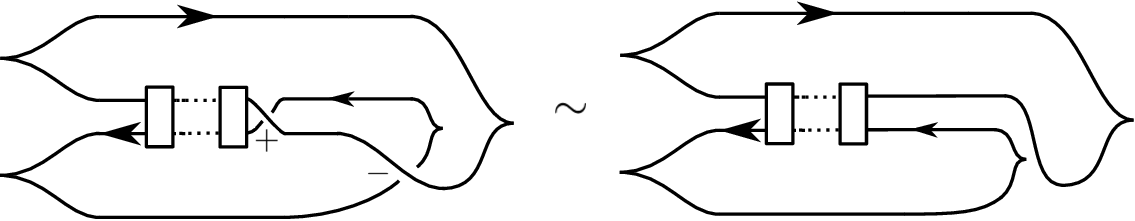}
\end{center}
\caption{
The proof of Theorem~\ref{thm:fillable1} for $s=2$. 
A box in this picture represents a positive crossing. 
The front projection $\Delta $ of the Legendrian link $\Lambda $ is the left. 
By using the Legendrian version of the Reidemeister move II, we obtain the right diagram $\Delta '$. }
\label{fig:s=2}
\end{figure}
\begin{cor}\label{cor:almost-equality}
Let $D$ be an almost positive knot diagram of a knot $K$ with the property $(P2)$. 
Then, we obtain 
\begin{align*}
\operatorname{TB}(K)+1=2\tau(K)=s(K)=2g_4(K)=2g_{3}(K)
&=-\max \operatorname{deg}_{v} P_{K}(v,z) \\
&=2g_{3}(D)-2,  
\end{align*}
where $g_{3}(D)$ is the genus of the Seifert surface obtained from $D$ by Seifert's algorithm. 
\end{cor}
\begin{proof}
Stoimenow \cite[Corollary $5$ and the proof of Theorems $5$ and $6$]{stoimenow1} proved that $g_{3}(K)=g_{3}(D)-1$. 
The author \cite{tagami5} proved that $g_4(K)=g_3(K)$ for any almost positive knot $K$. 
Hence, by Corollary~\ref{cor:fillable-equality}, we finish the proof. 
We remark that we obtain a $2$-graded normal ruling of $\Lambda $ by switching all positive crossings except the positive crossing given in the property $(P2)$, where $\Lambda $ is the exact Lagrangian fillable Legendrian knot constructed in the proof of Theorem~\ref{lem:fillable1}. 
\end{proof}
%
\begin{rem}
The author \cite[Remark~3.1]{tagami5} conjectured that any almost positive diagram of an almost positive knot satisfies (P$1$). 
However, it is false. In fact, it is known that $10_{145}$ is almost positive. On the other hand, $10_{145}$ has an almost positive diagram satisfying (P$2$) (see Figure~\ref{fig:picture1}).
In \cite[Theorem~1.4]{stoimenow4} Stoimenow proved that 
there exist almost positive knots with either none or all of
their almost positive diagrams having minimal genus. 
More precisely, Stoimenow proved that there are two almost positive knots $K_{1}$ and $K_{2}$ such that any almost positive diagram $D_{1}$ of $K_1$ satisfies (P$1$) and any almost positive diagram $D_{2}$ of $K_2$ satisfies (P$2$). 
By \cite[Corollary~5]{stoimenow1}, we have $g(D_{1})=g_{3}(K_{1})$ and $g(D_{2})-1=g_{3}(K_{2})$. 
This is the negative answer to \cite[Question~3]{stoimenow1} which asks whether any almost positive link has an almost positive diagram of minimal genus. 
\end{rem}
%
%
%
\section{Non-positive, Lagrangian fillable and strongly quasipositive knots}\label{sec:infinite}
The Lagrangian fillabilities of knots imply their quasipositivities. 
On one hand, Hayden and Sabloff \cite{Hayden-Sabloff} mentioned that Lagrangian fillability and strongly quasipositivity are independent conditions. 
The most famous class of Lagrangian fillable and strongly quasipositive knots are positive knots. 
Then, it is a natural question whether any Lagrangian fillable and strongly quasipositive knot is a positive knot. 
In this section, we give infinitely many almost positive (in particular, non-positive), Lagrangian fillable and strongly quasipositive knots. 
\begin{thm}\label{thm:almost-lagrangian}
For any $n\in\mathbf{Z}_{>0}$, the knot $K_{n}$ depicted in Figure~\ref{fig:Kn2} is almost positive, exact Lagrangian fillable and strongly quasipositive knot. 
\end{thm}
\begin{proof}
Stoimenow \cite[Example~6.1]{stoimenow3} proved that $K_{n}$ is almost positive. 
Abe and the author \cite[Figure~17]{abe-tagami1} gave a Seifert surface of $K_n$ which is represented by a Murasugi sum of some quasipositive surfaces. 
By Rudolph's work \cite{Rudolph2}, such a surface is a quasipositive, in particular, $K_n$ is strongly quasipositive. 
Finally, we prove that $K_n$ is exact Lagrangian fillable. 
By Figure~\ref{fig:Kn5}, the knot $K_{n}$ has an almost positive diagram satisfying (P$2$). 
By Theorem~\ref{thm:fillable1}, it is exact Lagrangian fillable. 
\end{proof}
\begin{figure}[h]
\begin{center}
\includegraphics[scale=0.86]{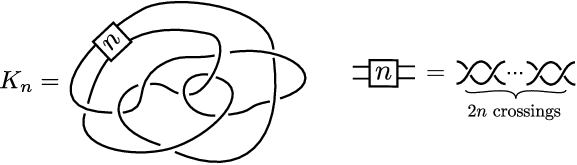}
\end{center}
\caption{An almost positive knot introduced by Stoimenow \cite[Example~6.1]{stoimenow3}. }
\label{fig:Kn2}
\end{figure}
\begin{figure}[h]
\begin{center}
\includegraphics[scale=0.9]{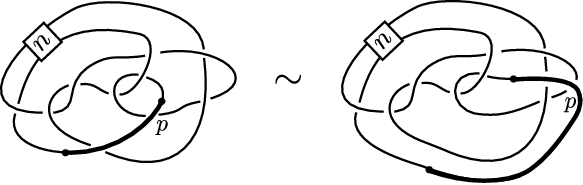}
\end{center}
\caption{The knot $K_n$ has an almost positive diagram satisfying (P$2$) (right).  The crossing $p$ is the negative crossing. 
}
\label{fig:Kn5}
\end{figure}

\begin{rem}
Recently, Feller, Lewark and Lobb \cite{Feller-Lewark-Lobb} proved that almost positive links are strongly quasipositive. Their result gives a positive answer to a question given by Stoimenow \cite[Question~4]{stoimenow1}. 

\end{rem}

\section{Further discussion}\label{sec:discussion}
In this section, we consider the positivity and Lagrangian fillability of links. 
It is known that positive links are homogeneous and strongly quasipositive (see \cite{homogeneous, nakamura1, Rudolph3}) and the converse is also true (see \cite{baader} and see also
 \cite{abe2, abe-tagami1}). 
Hayden and Sabloff \cite{Hayden-Sabloff} proved that positive links are exact Lagrangian fillable, and Lagrangian fillable links are quasipositive. 
\par
A'Campo \cite{ACampo} defined divide links. 
Gibson and Ishikawa \cite{Gibson-Ishikawa} constructed free divide links as an extension of divide links. 
Kawamura \cite{Kawamura3} defined the class of graph divide links, which is an extension of the class of free divide links, and proved that they are quasipositive. 
We note that the quasipositivity of free divide links was proved by Kawamura \cite{Kawamura4} before \cite{Kawamura3}. 
Abe and the author \cite[Lemma~3.2]{abe-tagami2} proved that the original divide links are strongly quasipositive. 
Tomomi Kawamura taught the author that this fact had been proved by Mikami Hirasawa (see also \cite[Remark~6.9]{Kawamura3}). 
Ishikawa \cite{Ishikawa1} proved that the maximal Thurston-Bennequin number of any graph divide link is equal to its slice Euler characteristic. 
This means graph divide links satisfy a necessary condition to be Lagrangian fillable. 
\par
Hence we obtain the following, where $P$ stands for positive, $LF$ Lagrangian fillable, $H$ homogeneous, $QP$ quasipositive, $SQP$ strongly quasipositive, and $Div$ divide: 
\[
\begin{matrix}
 \{H \text{ links}\} &&\{Div \text{ links}\}&\subset&\{\text{graph }Div \text{ links}\}\\
\cup  &&\cap&&\cap\\
\{ H \text{ and } SQP \text{ links} \}    & \subset & \{SQP \text{ links} \} &\subset & \{ QP \text{ links}\}.  \\
||&&  &  &\cup \\
\{P \text{ links} \}&   \subset &\{\text{exact\ } LF \text{ links} \} &\subset &\{LF \text{ links}\}
\end{matrix}
\]
Then, we can consider the following questions: 
\begin{question}
Are there non-positive and non-almost positive links in the set $\{LF \text{ links} \}\cap \{SQP \text{ links} \}$?
\end{question}
\begin{question}
Are there non-positive links in the set $\{LF \text{ links} \}\cap \{H \text{ links} \}$? 
\end{question}
\begin{question}
Is the set $\{\text{(graph)\ }Div \text{ links} \}$ contained in $\{LF \text{ links} \}$? 
\end{question}
%
%
In Theorem~\ref{thm:almost-lagrangian}, we give infinitely many almost positive (in particular non-positive), Lagrangian fillable and strongly quasipositive knots. 
There are non-positive, non-almost positive and Lagrangian fillable links (for example $8_{21}$, which is a graph divide knot).  
The author does not know any examples of non-positive, non-almost positive, Lagrangian fillable and strongly quasipositive links. 
\par
On alternating and Lagrangian fillable knots, the following is proved by Cornwell, Ng and Sivek \cite{Cronwell-Ng-Sivek}. 
\begin{thm}[{\cite[Theorem~4.3]{Cronwell-Ng-Sivek}}]
An alternating knot is Lagrangian fillable if and only if it is a positive knot. 
\end{thm}
\begin{proof}
The "if" part has been proved by Hayden and Sabloff (Theorem~\ref{thm:Hayden-Sabloff}). 
\par
Let $K$ be an alternating and Lagrangian fillable knot. 
Let $D$ be a reduced alternating diagram of $K$ with $c_{-}$ negative crossings. 
It is sufficient to prove that $c_{-}=0$.  
Ng \cite{Ng1} prove that $TB(K)=-c_{-}-\sigma (K)-1$, where $\sigma(K)$ is the signature of $K$. 
It is known that $s(K)=-\sigma(K)$ for any alternating knot $K$. 
By Corollary~\ref{cor:fillable-equality}, $TB(K)=s(K)-1$ for any Lagrangian fillable knot $K$. 
Hence, we obtain 
\[
c_{-}=-TB(K)-\sigma (K)-1=-(s(K)-1)+s(K)-1=0. 
\]
\end{proof}
Finally, we give the table of Lagrangian fillable and non-alternating knots with up to $10$ crossings (up to mirror image) (see Table~\ref{table1}). 
\begin{table}[htb]
  \begin{tabular}{|c|c||c|c||c|c||c|c|} \hline
    name & LF & name& LF& name& LF & name& LF\\\hline 
     $8_{19}$&Yes&$10_{127}$&Yes$^{\ast}$&$10_{141}$&No&$10_{155}$&No\\
     $8_{20}$&No&$10_{128}$&Yes&$10_{142}$&Yes&$10_{156}$&No\\
     $8_{21}$&Yes$^{\ast}$&$10_{129}$&No&$10_{143}$&No&$10_{157}$&Yes$^{\ast}$\\ 
     $9_{42}$&No&$10_{130}$&No&$10_{144}$&No&$10_{158}$&No\\
     $9_{43}$&No&$10_{131}$&Yes$^{\ast}$&$10_{145}$&Yes&$10_{159}$&No\\
     $9_{44}$&No&$10_{132}$&No&$10_{146}$&No&$10_{160}$&No\\
     $9_{45}$&Yes$^{\ast}$&$10_{133}$&Yes$^{\ast}$&$10_{147}$&No&$10_{161}$&Yes\\
     $9_{46}$&Yes$^{\ast}$&$10_{134}$&Yes&$10_{148}$&No&$10_{162}$&No\\
     $9_{47}$&No&$10_{135}$&No&$10_{149}$&Yes$^{\ast}$&$10_{163}$&No\\
     $9_{48}$&No&$10_{136}$&No&$10_{150}$&No&$10_{164}$&No\\
     $9_{49}$&Yes&$10_{137}$&No&$10_{151}$&No&$10_{165}$&Yes$^{\ast}$\\
     $10_{124}$&Yes&$10_{138}$&No&$10_{152}$&Yes&&\\
     $10_{125}$&No&$10_{139}$&Yes&$10_{153}$&No&&\\
     $10_{126}$&No&$10_{140}$&Yes$^{\ast}$&$10_{154}$&Yes&&\\ \hline
  \end{tabular}
  \caption{The Lagrangian fillability of non-alternating knots with up to $10$ crossings. 
For example, $8_{19}$ or its mirror is Lagrangian fillable. Neither $8_{20}$ nor its mirror is Lagrangian fillable. 
To prove ``Yes$^{\ast}$", we find front projections with maximal Thurston-Bennequin numbers and use Theorem~\ref{thm:tool}. 
To find such diagrams, we refer to \cite{knot_info} and \cite{legendrian-atlas}. }
  \label{table1}
\end{table}
\\
\ \\
\noindent{\bf Acknowledgements: }  
The author would like to thank Tomomi Kawamura for helpful comments on divide links. 
The author also thanks the referee for a careful reading and helpful comments. 
This work was supported by JSPS KAKENHI Grant numbers JP16H07230 and JP18K13416. 
%

\bibliographystyle{amsplain}
\bibliography{tagami}
\end{document}